\documentclass[12pt]{amsart}

\usepackage{amsmath, palatino}
\usepackage{amsthm}
\usepackage{enumerate}
\usepackage{amssymb}
\usepackage{graphicx}
\usepackage[utf8]{inputenc}
\usepackage[all]{xy}

\theoremstyle{plain}
\newtheorem{lemma}{Lemma}[section]
\newtheorem{theorem}[lemma]{Theorem}

\newtheorem{proposition}[lemma]{Proposition}

\newtheorem*{theorem*}{Theorem}
\newtheorem*{definition*}{Definition}

\newtheorem{remark}[lemma]{Remark}

 \newcommand{\CC}{\mathrm{C}}

 \newcommand{\V}{\mathrm{V}}
 \newcommand{\T}{\mathbb{T}}
 \newcommand{\N}{\mathbb{N}}

 \newcommand{\C}{\mathbb{C}}
 
 \newcommand{\K}{\mathrm{K}}

 \newcommand{\fa}{f\circ\varphi}
 \newcommand{\ga}{g\circ\varphi}
 \newcommand{\wa}{w\circ\varphi}

 \newcommand{\Her}{\mathrm{Her}}

\DeclareMathOperator{\Cu}{Cu} \DeclareMathOperator{\Lsc}{Lsc}
 \DeclareMathOperator{\supp}{supp}

\begin{document}

\title{Recovering the Elliott invariant from the Cuntz semigroup}
\date{\today}
\author{Ramon Antoine}\address{RA, FP \& LS: Departament de Matem\`atiques, Universitat Aut\`onoma de Barcelona, 08193 Bellaterra, Barcelona, Spain}\email{ramon@mat.uab.cat, perera@mat.uab.at, santiago@mat.uab.cat}
\author{Marius Dadarlat}\address{MD: Department of Mathematics, Purdue University, West Lafayette, IN 47907, USA}\email{mdd@math.purdue.edu}
\author{Francesc Perera}
\author{Luis Santiago}\curraddr{(LS) Department of Mathematics, University of Oregon, Eugene OR 97403, USA}\email{luissant@gmail.com}

\begin{abstract}
Let $A$ be a simple, separable C$^*$-algebra of stable rank one. We
prove that the Cuntz semigroup of $\CC(\T,A)$ is determined by its
Murray-von Neumann semigroup of projections and a certain semigroup
of lower semicontinuous functions (with values in the Cuntz
semigroup of $A$). This result has two consequences. First,
specializing to the case that $A$ is simple, finite, separable and
$\mathcal Z$-stable, this yields a description of the Cuntz
semigroup of $\CC(\T,A)$ in terms of the Elliott invariant of $A$.
Second, suitably interpreted, it shows that the Elliott functor and
the functor defined by the Cuntz semigroup of the tensor product
with the algebra of continuous functions on the circle are naturally
equivalent.
\end{abstract}

\maketitle

\section*{Introduction}

The Cuntz semigroup $\Cu(A)$ of a C$^*$-algebra $A$ is intimately
related to the classification program of simple, separable, and
nuclear algebras. This is a semigroup built out of equivalence
classes of positive elements in the stabilization of the algebra $A$
much in an analogous way as the projection semigroup $V(A)$ is, and
comes equipped with an order that is not algebraic, except for
finite dimensional algebras. One order property -- almost unperforation -- plays
a significant role in classification of such algebras up to isomorphism (see
\cite{toms}). This property is
equivalent to strict comparison, which allows to determine the order
in the semigroup by means of traces.

The Elliott conjecture predicts the existence of a $\K$-theoretic
functor $\mathrm{Ell}$ such that, for unital, simple, separable,
nuclear C$^*$-algebras $A$ and $B$ in a certain class, isomorphism
between $\mathrm{Ell}(A)$ and $\mathrm{Ell}(B)$ can be lifted to a
$^*$-isomorphism of the algebras. The concrete form of the invariant
(known as the Elliott invariant) for which this conjecture has had
tremendous success is the following:
\[
\mathrm{Ell}(A)=((\K_0(A),\K_0(A)^+,
[1_A]),\K_1(A),\mathrm{T}(A),r)\,,
\]
consisting of (ordered) topological $\K$-Theory, the trace simplex,
and the pairing between $\K$-Theory and traces given by evaluating a
trace at a projection (see, e.g. \cite{elliottoms}) (The category
where the said invariant sits will be described later.)

It is possible (and generally agreed) that the largest class for
which classification in its original form (i.e. using the Elliott
invariant as above) may hold consists of those algebras that absorb
the Jiang-Su algebra $\mathcal Z$ tensorially. Indeed, $\mathcal
Z$-stability springs into prominence as a necessary condition for
classification to hold (under the assumption of weak unperforation
on $\K_0$; see \cite{gongjiangsu}). This property of being $\mathcal
Z$-stable stands out as a regularity property for C$^*$-algebras,
together with finite decomposition rank and the condition of strict
comparison alluded to above. Among separable, simple, nuclear
C$^*$-algebras, a conjecture of Toms and Winter (see \cite{tomsprep}, and also
\cite{winterzach}) asserts that these three conditions are
equivalent.

The linkage between the Elliott invariant and the Cuntz semigroup
has been explored in a number of papers (see, e.g. \cite{bpt},
\cite{elliottrobertsantiago}, \cite{pertoms}, \cite{Tikuisis}). One
of the main results in \cite{bpt} recovers the Cuntz semigroup from
the Elliott invariant in a functorial manner, for the class of
simple, unital, $\mathcal Z$-stable algebras. Tikuisis shows, in
\cite{Tikuisis}, that the Elliott invariant  is equivalent to the
invariant $\Cu(\CC(\T,\cdot))$, for simple, unital, non-type I ASH
algebras with slow dimension growth (which happen to be $\mathcal
Z$-stable, as follows from results of Toms and Winter (\cite{atoms},
\cite{winterinv})). One of our main results in this paper confirms
that this equivalence can be extended to all simple, separable,
finite $\mathcal Z$-stable algebras. Thus, from a functorial point
of view and related to the Elliott conjecture, we prove the
following:

\begin{theorem*}
Let $A$ be a simple, unital, nuclear, finite C$^*$-algebra that
absorbs $\mathcal Z$ tensorially. Then:
\begin{enumerate}[{\rm (i)}]

\item There is a functor which recovers the Elliott invariant
$\mathrm{Ell}(A)$ from the Cuntz semigroup $\Cu(\CC(\T,A))$.

\item Viewing the Elliott invariant as a functor from the
category of C$^*$-algebras to the category Cu (where the Cuntz
semigroup naturally lives), there is a natural equivalence of
functors between $\mathrm{Ell}(\cdot)$ and $\Cu(\CC(\T,\cdot))$.
\end{enumerate}

\end{theorem*}

Since the Cuntz semigroup is a natural carrier of the ideal
structure of the algebra, it is plausible to expect that the object
$\Cu(\CC(\T,\cdot))$ may be helpful in the classification of
non-simple algebras.

The natural transformation that yields the equivalence of functors
in the theorem above is described in Section 4, and is based on
describing the Cuntz semigroup of $\CC(\T,A)$ for any simple,
separable, unital C$^*$-algebra of stable rank one. This is carried
out in Sections 2 and 3, and is done in terms of the Murray-von
Neumann semigroup of projections of $\CC(\T,A)$ together with a the
subsemigroup of the so-called non-compact lower semicontinuous
functions with values in $\Cu(A)$. Some of the methods used are
similar to the ones in \cite{aps}.

\section{Notation and Preliminaries}

We briefly recall the construction of the Cuntz semigroup and the
main technical aspects that we shall be using throughout the paper.
As a blanket assumption, $A$ will be a separable C$^*$-algebra.

Given positive elements $a$, $b$ in $A$, we say that $a$ is
\emph{Cuntz subequivalent to} $b$, in symbols $a\precsim b$, if
there is a sequence $(x_n)$ in $A$ such that $x_nbx_n^*\to a$ in
norm. The antisymmetrization $\sim $ of the relation $\precsim$ is
referred to as \emph{Cuntz equivalence}.

The \emph{Cuntz semigroup} of $A$ is defined as
\[
\Cu(A)=(A\otimes\mathcal K)_+/\sim\,.
\]
Denote the class of a positive element $a$ by $[a]$. Then $\Cu(A)$
is ordered by $[a]\leq [b]$ if $a\precsim b$, and it becomes an
abelian semigroup with addition given by
$[a]+[b]=[\left(\begin{smallmatrix} a & 0 \\ 0 & b
\end{smallmatrix}\right)]$.

As it was proved in \cite{CEI}, there exists a category of ordered semigroups, termed Cu,
with an enriched structure, such that the assignment $A\to\Cu(A)$
defines a sequentially continuous functor. We define this category
below.

In an ordered semigroup $S$, we say that $x$ is \emph{compactly
contained} in $y$ if, whenever there is an increasing sequence
$(z_n)$ with $y\leq\sup z_n$, there is $m$ such that $x\leq z_m$.
This is denoted by $x\ll y$ (see \cite{scottetal}). If $x\ll x$, we
say that $x$ is \emph{compact}. An increasing sequence $(x_n)$ is
termed \emph{rapidly increasing} provided that $x_n\ll x_{n+1}$ for
every $n$.

Define Cu to be the category whose objects are positively ordered
(abelian) semigroups for which: (i) every increasing sequence has a
supremum; (ii) every element is a the supremum of a rapidly
increasing sequence; and (iii) suprema and $\ll$ are compatible with
addition. Maps in Cu will be those semigroup maps that preserve
addition, order, suprema, and $\ll$. We shall be using repeatedly
the fact that, for any positive element $a$,
\[
[a]=\sup\limits_{n\to\infty} [(a-1/n)_+]\,,
\]
as follows from \cite[Proposition 2.4]{Rorfunct} (see also
\cite{KirchRor}). We shall be frequently using that, if $a$ and $b$
are positive elements and  $\|a-b\|<\epsilon$, then there is a
contraction $c$ in $A$ such that $(a-\epsilon)_+=cbc^*$, so in
particular $(a-\epsilon)_+\precsim b$ (see \cite[Lemma
2.2]{Kirchberg-Rordam} and also \cite[Proposition 2.2]{Rorfunct}).

For a compact space $X$ and a semigroup $S$ in the category Cu, we
shall use $\Lsc(X,S)$ to denote the ordered semigroup of all lower
semicontinuous functions from $f\colon X\to S$, with pointwise order
and operation. (Here, $f$ is lower semicontinuous if, for any $x\in
S$, the set $\{t\in X\mid x\ll f(t)\}$ is open in $X$.)

If $A$ is a C$^*$-algebra and $X$ is a one dimensional compact
Hausdorff topological space, then there is a natural map:
\[\begin{array}{ccc} \alpha\colon \Cu(\CC(X,A)) & \longrightarrow &
\Lsc(X,\Cu(A))
\\  \mbox{} x & \longmapsto & \mbox{} \hat{x}
\end{array} \]
where, if $x=[f]$, then $\hat x(t)=[f(t)]$. It is proved in
\cite[Theorem 5.15]{aps} that $\Lsc(X,\Cu(A))$ equipped with the
point-wise order and addition is a semigroup in $\Cu$, and that
$\alpha$ is a well defined map in $\Cu$  which is an order embedding
in case $A$ has stable rank one and  $K_1(I)=0$ for all ideals of
$A$. Furthermore, $\alpha$ is surjective provided it is an order
embedding (and thus an order isomorphism).


\section{The Cuntz semigroup of $\CC([0,1],A)$ for a simple algebra $A$}

In this section, we prove that if $A$ is a simple C$^*$-algebra with
stable rank one, then the Cuntz semigroup of $\CC([0,1],A)$ is
order-isomorphic to $\Lsc([0,1],\Cu(A))$, thus obtaining the same
result as in \cite[Theorem 2.1]{aps} for a simple algebra, but
without requiring that $\K_1(A)=0$. The key point in the argument is
based on the fact that, for certain continuous fields of
C$^*$-algebras, unitaries from fibres can be lifted to unitaries in
the algebra.


\begin{lemma}\label{juntarunitaris}
 Let $A$ be a unital continuous field of C$^*$-algebras over $X=[0,1]$ and let $u,v\in \mathrm{U}(A)$. If $u(t_0)\sim_h v(t_0)$ for some $t_0\in (0,1)$, then
there exists $w\in \mathrm{U}(A)$ such that $w(0)=u(0)$ and
$w(1)=v(1)$.
\end{lemma}

\begin{proof}
 Since $u(t_0)\sim_h v(t_0)$, we have $(vu^*)(t_0)\sim_h 1_{A(t_0)}$. Therefore, there exists a unitary $\tilde w\in \mathrm{U}_0(A)$,
such that $\tilde w(t_0)=vu^*(t_0)$. Consider a continuous path
$w_s$ of unitaries in $\mathrm{U}_0(A)$ such that $w_0=1_{A}$ and
$w_{t_0}=\tilde w$. Let us define the following element $w\in
\prod_{t\in [0,1]}A(t)$ given by
\[ w(s):= \left\{\begin{array}{lr} w_s(s)u(s) & \text { if } s\leq t_0. \\ v(s) & \text{otherwise} \end{array} \right.\]

Clearly $w(0)=u(0)$ and $w(1)=v(1)$. Since $A$ is a continuous field
of C$^*$-algebras, to prove $w\in A$ it is enough to find, for each
$t\in [0,1]$ and $\epsilon>0$, a neighborhood $V_t$ of $t$ and an
element $z\in A$ such that $\|w(s)-z(s)\|<\epsilon$ for all $s\in
V$.

This is obvious if $t\in (t_0,1]$. If $t\in [0,t_0)$, and
$\epsilon>0$, there exists a neighborhood $V_t$ such that for all
$s,s'\in V_t$, $\| w_{s}-w_{s'}\|<\epsilon$ (since $w_s$ is a
continuous path). Hence, considering the element $z=w_tu\in A$, we
have, for all $s\in V_t$,
\[\| w(s)-(w_tu)(s)\|=\|w_s(s)u(s)-w_t(s)u(s)\|\leq \| w_s(s)-w_t(s)\|\cdot \|u(s)\|<\epsilon. \]

Now for $t=t_0$, since $\|(w_{t_{0}}u)(t_0)-v(t_0)\|=0$ and by the
continuity of the norm in $A$, there exists a neighborhood $V_{t_0}$
such that $\|(w_{t_0}u)(s)-v(s)\|<\epsilon$ for all $s\in V_{t_0}$,
and furthermore we can choose $V_{t_0}$ such that
$\|w_{s}-w_{t_0}\|<\epsilon$. Now, with a similar argument as above,
we are done taking $z=w_{t_0}u$.
\end{proof}


Given a $C^*$-algebra $A$ and a hereditary subalgebra $B\subseteq
\CC(X,A)$, $B$ becomes a continuous field of C$^*$-algebras over $X$
whose fibres $B_x$ can be identified with hereditary subalgebras of
$A$. If $A$ is simple, then for all $x\in X$ such that $B_x\neq 0$,
the inclusion $i_x\colon B_x\to A$ induces an isomorphism
\[ \begin{array}{rccc}
    (i_{x})_*\colon & \K_1(B_{x}) & \rightarrow & \K_1(A).
   \end{array}
\]

If $A$ has stable rank one,  then
$\K_1(A)=\mathrm{U}(A^\sim)/\mathrm{U}_0(A^\sim)$ and elements can
be identified with connected components of unitaries in
$\mathrm{U}(A^\sim)$, which we denote by $[v]_A$. Hence, for all
$x$, $B_x$ will also have stable rank one and
$(i_x)_*([v]_{B_x})=[i_x^\sim(v)]_{A}=[v]_A$ where $i_x^\sim\colon
B_x^\sim\to A^\sim$ denotes tha natural extension to the
unitizations.

Let $D_B=B+\CC(X)\cdot 1_{\CC(X,A^\sim)}\subseteq \CC(X,A^\sim)$.
Then $D_B$ is a unital continuous field of C$^*$-algebras whose
fibres $D_B(x)\cong B_x+\C\cdot 1_{\CC(X,A^\sim)}(x)$. Assuming $A$
is stable, we have $1_{A^\sim}=1_{\CC(X,A^\sim)}(x)\not\in
B_x\subseteq A$, hence $D_B(x)\cong B_x^\sim$. Observe furthermore
that the following diagram commutes:
\[
\xymatrix{ D_B=B+\CC(X)\cdot 1_{\CC(X,A^\sim)}\ar@{->>}_{\pi_x}[d]\ar@{^{(}->}[r] & \CC(X,A) + \CC(X)\cdot1_{\CC(X,A^\sim)}\ar@{->>}^{\pi_x}[d] \\
D_B(x)=B_x+\C\cdot 1_{\CC(X,A^\sim)}(x)\ar@{^{(}->}^-{i_x^\sim}[r] &
A + \C\cdot 1_{\CC(X,A^\sim)}(x)=A^\sim}
\]
Hence we will assume, $D_B(x)=B_x^\sim\subseteq A^\sim
=\CC(X,A^\sim)(x)$.

\begin{proposition}\label{pujarunitari}
 Let $A$ be a simple C$^*$-algebra with stable rank one and let $X$ be a finite graph. Suppose $B$ is a hereditary subalgebra of
$\CC(X,A)$ such that $B_x\neq 0$ for all $x\in X$. Let $(i_x)_*$  
denote the induced isomorphisms. Let  $x_0,\dots,x_n \in X$ and
$u_i\in \mathrm{U}(B_{x_i}^\sim)$ for $i=0,\dots,n$. If
$(i_{x_k})_*([u_k])=(i_{x_l})_*([u_l])$ for all $k,l$, then there
exists $u\in \mathrm{U}(D_B)$ such that $u(x_i)=u_i$ for
$i=0,\dots,n$.
\end{proposition}

\begin{proof}
Let us view $X$ as a $1$-dimensional simplicial complex where its $0$-skeleton is
$X_0=\{x_0,\dots,x_n,\dots,x_m\}$ (possibly adding
vertices to $X_0$ as new points $x_i$). To define a unitary $u\in
\mathrm{U}(D_B)$, it is enough to define it first in $X_0$ and then
in each of the edges  of the $1$-skeleton, provided the values in
the boundary match the corresponding values in $X_0$. Since $X_0$ is
a finite set of points, $u$ can be easily defined pointwise (see
below) choosing, for all $i=1,\dots,n$, $u(x_i)=u_i$. Therefore, in
order to define $u$ for the 1-skeleton we can reduce to the case
$X=[0,1]$ and $x_0=0, x_1=1$ with unitaries $u_0,u_1$ such that
$[u_1]=(i_{x_1})_*^{-1}(i_{x_0})_*[u_0]$.

Let us choose, for the remaining $x\in (0,1)$, unitaries  $u_x\in
\mathrm{U}(B_x^\sim)$ such that
\[ [u_x]_{B_x}=(i_x)_*^{-1}(i_{x_0})_*[u_0]_{B_{x_0}},\]
and hence such that $[i_{x_0}(u_0)]_A=[i_{x}(u_x)]_A$ which means
that $u_x$ and $u_0$ are connected in $\mathrm{U}(A^\sim)$.

For each $x\in X$ we can find an open neighborhood $V_x$ such that
$u_x=v_x(x)$ for some $v_x\in D_B$ and ${v_x}|_{ {\overline{V_x}}}$
is a unitary. Since $X$ is compact, we can find a finite number of
such neighborhoods $V_{x_0}:=V_0,\dots,V_r:=V_{x_1}$ covering $X$.
Furthermore by restricting the  $V_x$'s to be open intervals we can
assume that the resulting cover has multiplicity 1 and denote
$V_i\cap V_{i+1}=(a_i,b_i)$ for $i=1,\dots, r-1$
($a_i<b_i<a_{i+1}$). For $i=1,\dots,r-1$, let us assume
$V_i=V_{y_i}$ for some $y_i\in X$, $y_0=x_0=0$ and $y_r=x_1=1$.

For each $i=0,\dots,r-1$ choose $z_{i}\in (a_i,b_i)$. Since
$D_B\subseteq \CC(X,A^\sim)$, both $v_{y_i}\mid_{\overline{V_i}}$
and $v_{y_{i+1}}\mid_{\overline{V_{i+1}}}$ are paths of unitaries in
$A^\sim$. Hence in $A^\sim$ we have
\[v_{y_i}(z_{i})\sim_h v_{y_i}(y_i)=u_{y_i}\sim_h u_0 \sim_h u_{y_{i+1}}= v_{y_{i+1}}(y_{i+1})\sim_h v_{y_{i+1}}(z_{i}).\]
This implies
$(i_{z_{i}})_*[v_{y_{i}}(z_{i})]=(i_{z_{i}})_*[v_{y_{i+1}}(z_{i})]$,
but since $(i_{z_{i}})_*$ is an isomorphism, we obtain
$v_{y_{i+1}}(z_{i})\sim_h v_{y_{i}}(z_{i})$ in $B(z_{i})^\sim$. Now,
using Lemma~\ref{juntarunitaris}, we can construct a unitary $w_{i}$
in $D_B(\overline{V_i\cap V_{i+1}})^\sim=D_B([a_{i},b_{i}])^\sim$
such that $w_{i}(a_{i})=v_{y_i}(a_{i})$ and
$w_{i}(b_{i})=v_{y_{i+1}}(b_{i})$.

Therefore, defining $v\in D_B$ as the following element in
$\prod_{x\in X}B_x^\sim$
\[v(x)=\left\{ \begin{array}{lr} v_{y_i}(x) & \text{ if } x\in V_{i}\setminus(V_{i-1}\cup V_{i+1}) \\ w_{i}(x) & \text{ if } x\in V_i\cap V_{i+1} \end{array}\right.\]
we obtain an element in $D_B$, which is furthermore a unitary and
$v(0)=u_0, v(1)=u_1$.

\end{proof}

\begin{remark}
\label{rem:esa} Observe that, in the particular case of only one
point $x_0\in X$, the Proposition states that the map
$\mathrm{U}(D_B)\to \mathrm{U}(B_{x_0}^\sim)$ is surjective, and
thus we can lift unitaries from each fibre.
\end{remark}


Let $X$ be a locally compact, Hausdorff space. Suppose $A$ is a
continuous field of C$^*$-algebras over $X$ and $a\in A$. We denote
by $\supp(a)=\{x\in X\mid a(x)\neq 0\}$. Observe that, since the
assignment $x\to \|a(x)\|$ is continuous, $\supp(a)$ is an open
subset of $X$. If $Y\subseteq X$ is a closed subset of $X$, and
$a\in A$, then $a|_{ Y}$ denotes the image of $a$ by the projection
$\pi_Y:A\to A(Y)$.

\begin{lemma}\label{lematonto}
Let $A$ be a continuous field of C$^*$-algebras over a space $X$ and
let $a,b\in A_+$.
\begin{enumerate}
 \item If $X=\sqcup_{i=1}^r X_i$ is a finite disjoint union of open sets,
 then $a\precsim b$ if and only if $a|_{X_i} \precsim b|_{X_i}$ for $i=1,\dots,r$.

\item If $b|_{ K}\precsim a|_{ K}$ for some $K$ such that $\supp(b)\subseteq K\subseteq \supp(a)$, then $a\precsim b$.
\end{enumerate}
\end{lemma}

\begin{proof}

(i) Is clear since $A\cong \oplus_{i=1}^r A(X_i)$. Let us prove
(ii). Suppose $b|_{ K}\precsim a|_{ K}$ as in the statement. Given
$\epsilon>0$, we can find $d\in A$ such that
$\|b(x)-d(x)a(x)d^*(x)\|<\epsilon$ for all $x\in K$, but since $A$
is a continuous field of C$^*$-algebras, this is valid in an open
set $K\subseteq U$,
\begin{equation}\label{aaK} ||b(x)-dad^*(x)||<\epsilon \quad \text{ for all }x\in U.\end{equation}

Now, since $K\cap U^c=\emptyset$ we can consider a continuous
function $\lambda\colon X\to [0,1]$ such that $\lambda|_{ K}=1$ and
$\lambda|_{ U^c}=0$. If $x\in \supp(b)\subseteq K\subseteq U$, then
\[\|b(x)-(\lambda d)a(\lambda
d)^*(x))\|=\|b(x)-d(x)a(x)d^*(x)\|<\epsilon,\] by \eqref{aaK}, and
if $x\not\in U$ then $||b(x)-(\lambda d)a (\lambda d)^*(x)||=0$.
Finally, if $x\in U\setminus \supp(b)$, then $b(x)=0$, and
\[||b(x)-(\lambda d)a(\lambda d)^*(x)||=||\lambda^2b(x)-(\lambda
d)a(\lambda d)^*(x)||=|\lambda^2(x)|\cdot \|b(x)-dad^*(x)\|
<\epsilon,\] again by \eqref{aaK}. Hence, since $||b-(\lambda
d)a(\lambda d)^*||=\sup_{x\in X}||b(x)-(\lambda d)a(\lambda
d)^*(x)||<\epsilon$, we obtain $b\precsim  a$.

\end{proof}

Recall that if $X$ is a locally compact Hausdorff topological space,
then the set $\mathcal O(X)$ consisting of open sets ordered by
inclusion is a continuous lattice. In the case $X$ is second
countable, we have that $U\ll V$ whenever there exists a compact set
$K$ such that $U\subseteq K\subseteq V$ (the countability condition
is needed since our definition of compact containment is only for
increasing sequences, and not arbitrary nets). In fact, $\mathcal
O(X)$ with union as addition is a semigroup in Cu, which can be
described as $\Lsc(X,\{0,\infty\})$ through the assignment $f\to
\supp(f)$ (since $\infty$ is a compact element in $\{0,\infty\}$ and
thus $\supp(f)=f^{-1}(\{\infty\})$ is an open set, and, by the same
argument, the characteristic function (in $\{0,\infty\}$) of any
open set, is lower semicontinuous). The following Lemma illustrates
the relation of Cuntz order in a continuous field of C$^*$-algebras
over $X$ with the ordered structure of $\mathcal O(X)$.

%

\begin{lemma}\label{suportwaybelow}
Let $A$ be a (stable) continuous field of C$^*$-algebras over a
compact Hausdorff space $X$ and $a,b\in A_+$ such that $[b]\leq
[a]$. Then $\supp(b)\subseteq \supp(a)$ and, if $[b]\ll [a]$ we have
$\supp(b)\ll \supp(a)$.
\end{lemma}

\begin{proof}
The first statement is obvious, let us suppose $[b]\ll [a]$ for some
$a,b\in A_+$. Since $\mathcal O(X)$ is in Cu, let us write
$\supp(a)=\cup_{i\geq 0} U_i$ for some $U_i\ll U_{i+1}$ (hence
$U_i\subseteq \overline{U_i} \subseteq U_{i+1}$). We can find, by
Urysohn's Lemma, continuous functions $\lambda_n\colon X\to [0,1]$
such that $\lambda_n(\overline{U_n})=1$ and
$\lambda_n(U_{n+1}^c)=0$. Since $X$ is compact we obtain $\lambda_n
a\to a$ and $\lambda_n a\leq \lambda_{n+1}a$, thus $[a]=\sup_n
[\lambda_n a]$. Now since $[b]\ll [a]$ we get $[b]\leq [\lambda_N
a]$ for some $N>0$ and therefore $\supp(b)\subseteq \supp(\lambda_N
a)\subseteq U_{N+1}\subseteq \overline{U_{N+1}}\subseteq \supp(a)$.
Hence $\supp(b)\ll \supp(a)$.
\end{proof}

\begin{theorem}
\label{thm:interval} Let $A$ be a C$^*$-algebra which is separable,
simple and has stable rank one. Then, the map $\alpha\colon
\Cu(\CC([0,1],A))\to \Lsc([0,1],\Cu(A))$ is an order isomorphism.
\end{theorem}

\begin{proof} We may assume $A$ is stable. Suppose $f,g\in \CC([0,1],A)$ are
such that $f(t)\precsim g(t)$. It is enough to prove that
$(f-\delta)_+\precsim g$ for all $\delta>0$, so let us first assume
that $[f]\ll [g]$. Hence, by Lemma~\ref{suportwaybelow} we have
$\supp(f)\ll \supp(g)$, and thus there exists a compact set $K$ such
that $\supp(f)\subseteq K\subseteq \supp(g)$. Since finite unions of
open intervals form a dense subset of $\mathcal O([0,1])$, and $K$
is compact, we may further assume that $K$ is a finite union of
closed intervals. Now by virtue of Lemma~\ref{lematonto} (i) and
(ii), we may finally assume that  $\supp(g)=[0,1]$.

Now the proof follows the lines of \cite[Theorem 2.1]{aps}. In
there, $\K_1(A)=0$ was assumed in order to lift unitaries from
$\Her(g(t))^\sim$ to $D_{\Her(g)}=\Her(g)+\CC(X)\cdot
1_{M(\CC(X,A))}$. From our argument in the previous paragraph we can
reduce to the case where $\Her(g(t))^{\sim}\neq 0$ for all $t$ and
then use Proposition~\ref{pujarunitari} with $B:=\Her(g)$ (see also
Remark \ref{rem:esa}).
\end{proof}


\section{The invariant $\mathrm{Cu}_{\mathbb T}(A)$}

In this section we give a complete description of the Cuntz
semigroup of $\CC(\T,A)$, for a simple, separable C$^*$-algebra $A$
that has stable rank one. We also show that, in the simple,
$\mathcal Z$-stable, finite case, the information it contains is
equivalent to that of the Elliott invariant (see next section and
also \cite{Tikuisis}).

We start with the following:

\begin{lemma}
\label{lem:spectrum} Let $A$ be a simple C$^*$-algebra with stable
rank one, and let $y\in A$ be a contraction. Let $\epsilon>0$ be
such that $\epsilon\in\sigma(yy^*)$, and let $B$ be a hereditary
subalgebra of $A$ with $yy^*\in B$. If $u$, $v$ are unitaries in
$B^\sim$, then there is $u_0\in U(B^\sim)$ with $[u_0]=[v]$ in
$\mathrm{K}_1(B)$, and
\[
\|uy-u_0y\|<5\sqrt{\epsilon}\,.
\]
\end{lemma}

\begin{proof}
We know that $\epsilon\in\sigma(yy^*)$, so there exists $0<c$ with
$\|c\|<2\epsilon$ and such that
\[
(yy^*-2\epsilon)_+\perp c\,,\text{ and } (yy^*-2\epsilon)_++c\leq
yy^*\,.
\]
Note that $c\in B$. Write $d=(yy^*-2\epsilon)_+$. As $c\neq 0$,
inclusion induces an isomorphism $\mathrm{K}_1(\overline{cAc})\cong
\mathrm{K}_1(B)$, so there is $w\in U(\overline{cAc}^{\sim})$ of the
form $1+a$, where $a\in \overline{cAc}$ such that $[w]\mapsto
[v]-[u]$.

Put $u_0=uw$, a unitary in $B^{\sim}$. Note that, in
$\mathrm{K}_1(B)$, we have $[u_0]=[u]+[w]=[u]+[v]-[u]=[v]$.

Next, choose $czc\in cAc$ such that $\|a-czc\|<\epsilon/2$. Then,
$\|w-(1+czc)\|=\|a-czc\|<\epsilon/2$, so $\|czc\|<\epsilon/2+2$.
Compute that
\[
u(1+czc)(d+c)=u(d+c+czc^2)=ud+uc+uczc^2\,,
\]
whence
\[
\|u(1+czc)(d+c)-u(d+c)\|=\|uczc^2\|<(2+\epsilon/2)2\epsilon=4\epsilon+\epsilon^2\,.
\]
Therefore
\begin{align*}
\|uyy^*-u_0yy^*\|\leq &
\|uyy^*-u(d+c)\|+\|u(d+c)-u_0(d+c)\|+\|u_0(d+c)-u_0yy^*\|\\
<& 4\epsilon+\|u(d+c)-u_0(d+c)\|\\
\leq & 4\epsilon
+\|u(d+c)-u(1+czc)(d+c)\|+\|u(1+czc)(d+c)-uw(d+c)\|\\
< & 8\epsilon+\epsilon^2
+\|u(1+czc-w)(d+c)\|<8\epsilon+\epsilon^2+\epsilon/2\,.
\end{align*}
Thus
\[
\|uy-u_0y\|^2=\|(u-u_0)yy^*(u-u_0)^*\|<2(8\epsilon+\epsilon^2+\epsilon/2)<19\epsilon\,,
\]
so that $\|uy-u_0y\|<\sqrt{19\epsilon}<5\sqrt{\epsilon}$.
\end{proof}

\begin{proposition}
\label{prop:noprojectioncase} Let $A$ be a simple, separable
$C^*$-algebra of stable rank one. Let $f$ and $g$ be elements in
$\CC(\T,A)$ such that $f$ is not equivalent to a projection, and $g$
is never zero. If $f(t)\precsim g(t)$ for all $t\in\T$, then
$f\precsim g$.
\end{proposition}
\begin{proof}
Since $f$ is not equivalent to a projection, zero is an isolated
point of $\sigma(f)$ and this implies, as
$\sigma(f)=\overline{\cup_{t\in \T}\sigma (f(t))}$, that for every
$n$, there is $t_n\in \T$ and $\lambda_n\in \sigma (f(t_n))$ with
$0<\lambda_n<1/2^n$. By compactness, and passing to a subsequence if
necessary, we may assume that $(t_n)$ converges to a point $t_0$. We
shall assume that the sequence $(t_n)$ is not eventually constant,
since otherwise (that is, $f(t_0)$ itself is not equivalent to a
projection), the argument is similar, and easier.

Let $\varphi\colon [0,1]\to \T$ be the map that $\varphi(0)=\varphi
(1)=t_0$. Since $f(t)\precsim g(t)$ for all $t$, this also holds
when composing with $\varphi$, so $(\fa)(s)\precsim (\ga)(s)$ for
all $s\in [0,1]$.

Let $0<\epsilon<1$. There exists $d\in A$ such that
$\|(\fa)(0)-d^*(\ga)(0)d\|<\epsilon$. There is then a neighbourhood
$U$ of $0$ and $1$ such that, with $h(s)=d$, we have
\[
\|(\fa-h^*(\ga)h)|_{U}\|<\epsilon\,.
\]
Write $U=[0,s_0)\cup(s_0',1]$, with $s_0<s_0'$. Now, there exists
$\epsilon'<\epsilon^2$, $s_1\in U$ and
$\lambda_{\varphi(s_1)}\in\sigma(f(\varphi(s_1)))$ such that
$\epsilon'<\lambda_{\varphi(s_1)}<\epsilon^2$, and we may assume
(without loss of generality) that $0<s_1<s_0$. Choose also
$s_0'<s_2<1$.

By Theorem \ref{thm:interval}, there exists $c\in \CC([0,1],A)$ such
that $\|\fa-c^*(\ga)c\|<\epsilon'/2$. By \cite[Lemma
2.2]{Kirchberg-Rordam}, there is a contraction $e\in \CC([0,1],A)$
such that, with $y_1=(\ga)^{1/2}ce$, we have
$((\fa)-\epsilon'/2)_+=y_1^*y_1$. If we let $y_2=(\ga)^{1/2}h$, we
have
\[
\|\fa-y_1^*y_1\|\leq\epsilon'/2<\epsilon'\,,\,\,\|\fa-y_2^*y_2\|<\epsilon\text{
and } y_iy_i^*\in\Her(\ga)\text{ for }i=1,2\,.
\]
By evaluating at the $s_i$, for $i=1,2$, we get
\[
\|(\fa)(s_i)-y_1^*y_1(s_i)\|<\epsilon'\text{ and
}\|(\fa)(s_i)-y_2^*y_2(s_i)\|<\epsilon\,,
\]
so we may apply \cite[Lemma 1.4]{aps} to find unitaries
\[
u_1'\in\Her((\ga)(s_1))^{\sim}\text{ and
}u_2\in\Her((\ga)(s_2))^{\sim}
\]
such that
\[
\|u_1'y_1(s_1)-y_2(s_1)\|<9\epsilon\text{ and
}\|u_2y_1(s_2)-y_2(s_2)\|<9\epsilon\,.
\]

Let $u_1''$ be a unitary such that
\[
[u_1'']=(i_{s_1})_*^{-1}\circ (i_{s_2})_*([u_2])\,.
\]
Since $\lambda_{\varphi(s_1)}\in\sigma ((\fa)(s_1))$, we have that
$0<\lambda_{\varphi(s_1)}-\epsilon'/2\in\sigma
(((\fa)(s_1)-\epsilon'/2)_+)=\sigma(y_1^*y_1(s_1))$, so
$\lambda_{\varphi(s_1})-\epsilon'/2\in \sigma (y_1y_1^*(s_1))$. By
Lemma \ref{lem:spectrum}, there is a unitary $u_1\in\Her
((\ga)(s_1))^{\sim}$ such that
\[
[u_1]=[u_1'']\text{ in }\mathrm{K}_1(\Her((\ga)(s_1)))
\]
and
\[
 \|u_1'y_1(s_1)-u_1y_1(s_1)\|<5\sqrt{\lambda_{\varphi(s_1)}-\epsilon'/2}
<5\epsilon\,.
\]
Thus
\[
\|u_1y_1(s_1)-y_2(s_1)\|\leq
\|u_1y_1(s_1)-u_1'y_1(s_1)\|+\|u_1'y_1(s_1)-y_2(s_1)\|<5\epsilon+9\epsilon
= 14\epsilon\,.
\]

By Proposition \ref{pujarunitari} there is a unitary $w\in
D_{\Her(g)}$ such that $w(\varphi(s_1))=u_1$ and
$w(\varphi(s_2))=u_2$.

%

Put $y_1'=(\wa)y_1$, and notice that
$\|\fa-(y_1')^*(y_1')\|=\|\fa-y_1^*y_1\|<\epsilon'<\epsilon$, and
also that
\[
\|y_1'(s_1)-y_2(s_1)\|=\|(w(\varphi(s_1))y_1(s_1)-y_2(s_1)\|=\|u_1y_1(s_1)-y_2(s_1)\|<14\epsilon
\]
and
\[
\|y_1'(s_2)-y_2(s_2)\|=\|w(\varphi(s_2))y_1(s_2)-y_2(s_2)\|=\|u_2y_1(s_2)-y_2(s_2)\|<9\epsilon\,.
\]
Therefore, there exists a neighbourhood $W\subset U$ of $s_1$ and
$s_2$, that neither contains $0$ nor $1$, with
\[
\|y_1'(s)-y_2(s)\|<14\epsilon\text{ for all }s\in W\,.
\]
Let $V=[0,s_1)\cup (s_2,1]$, and let $\mu_1$, $\mu_2$ be a partition
of unity associated to the covering $V\cup W$, $V^c\cup W$, and
consider the element
\[
z=\mu_1y_1'+\mu_2y_2\,.
\]
Note that $z(0)=y_2(0)=y_2(1)=z(1)$, so $z\in \CC(\T,A)$. Also
$zz^*\in\Her(g)$.

We need to estimate $\|f-z^*z\|$. It is enough to consider
$(f-z^*z)|_{W}$. Since
$(y_1'-z)|_{W}=(y_1'-\mu_1y_1'-\mu2y_2)|_{W}=\mu_2(y_1'-y_2)|_{W}$,
we see that $\|(y_1'-z)|_{W}\|\leq \|(y_1'-y_2)|_{W}\|<14\epsilon$.
Therefore, a standard argument shows that
\[
\|((y_1')^*y_1'-z^*z)|_{W}\|<28\epsilon\sqrt{1+\epsilon}<42\epsilon\,,
\]
whence
\[
\|(f-z^*z)|_{W}\|<\epsilon+42\epsilon=43\epsilon\,.
\]
This implies that $(f-43\epsilon)_+\precsim g$, and since
$\epsilon>0$ is arbitrary, it follows that $f\precsim g$ in
$\CC(\T,A)$, as desired.
\end{proof}

\begin{remark}
{\rm  In view of the previous result, the reader may wonder whether
if an element $f\in \CC(X,A)$ is not equivalent to a projection,
then there is some point $x\in X$ such that $f(x)$ is itself not
equivalent to a projection. We remark this is not true, as is seen
by taking, e.g. $X=[0,1]$, $p$ any non-zero projection in $A$,
$\lambda(t)=(1/2-t)_+$, and $f=\lambda p$. Then, clearly $f$ is
equivalent to a projection pointwise, but not globally.}
\end{remark}

\begin{proposition}
\label{prop:generalcase} Let $A$ be a simple, separable
$C^*$-algebra of stable rank one. Let $f$ and $g$ be elements in
$\CC(\T,A)$ such that $f$ is not equivalent to a projection. If
$f(t)\precsim g(t)$ for all $t\in\T$, then $f\precsim g$.
\end{proposition}
\begin{proof}
If $g$ is never zero, then the result follows from Proposition
\ref{prop:noprojectioncase}. We may therefore assume that, without
loss of generality, $g(1)=0$ (and then also $f(1)=0$).

Let $\varphi\colon [0,1]\to\T$ be the map such that
$\varphi(0)=\varphi(1)=1$. Since $\fa(s)\precsim\ga(s)$ for every
$s\in [0,1]$, it follows from Theorem \ref{thm:interval} that
$(\fa)\precsim(\ga)$. Let $\epsilon>0$. Find $c\in \CC([0,1],A)$
such that
\[
\|\fa-c(\ga)c^*\|<\epsilon/2\,.
\]
Since $f(1)=g(1)=0$, there is a neighbourhood $U$ of $0$ and $1$
such that $\|(\fa)|_{U}\|<\epsilon$ and
$\|(\ga)|_{U}\|<\epsilon/(2\|c\|^2)$. Let $\lambda\colon [0,1]\to\C$
be a continuous function such that $0\leq\lambda\leq 1$,
$\lambda|_{U^c}=1$, and $\lambda (0)=\lambda(1)=0$, and let
$d=\lambda^{1/2} c$, which defines an element in $\CC(\T,A)$. Then
$(\fa-\lambda c(\ga) c^*)|_{U^c}=(\fa-c(\ga)c^*)|_{U^c}$, and
\[
\|(\fa-\lambda c(\ga)c^*)|_{U}\|\leq
\|(\fa-c(\ga)c^*)|_{U}\|+\|(1-\lambda)|_{U}\|\|g\|\|c\|^2<\epsilon/2+\epsilon/2=\epsilon\,,
\]
whence $\|f-dgd^*\|<\epsilon$ so $(f-\epsilon)_+\precsim g$. Since
$\epsilon>0$ is arbitrary, this implies that $f\precsim g$, as was
to be shown.
\end{proof}

We are now ready to describe the Cuntz semigroup of $\CC(\T,A)$,
whenever $A$ is simple and has stable rank one. As $A$ is, in
particular, stably finite, this is also the case for $\CC(\T,A)$.
Thus, upon identification of $\V(\CC(\T,A))$ with its image in
$\Cu_{\T}(A)$, we have
\[
\Cu_{\T}(A)=\\V(\CC(\T,A))\sqcup\Cu_{\T}(A)_\mathrm{nc}\,,
\]
where $\Cu_{\T}(A)_\mathrm{nc}$ stands for the subsemigroup of non-compact
elements.

Observe that $\Cu_{\T}(A)\to\Lsc(\T,\Cu(A))$ sends compact elements
to compact elements. Using the arguments in \cite[Corollary
3.8]{aps}, those are the functions that take a constant value in
$\V(A)$.

If $X$ is a compact Hausdorff, connected space, and $S$ is a
semigroup in the category Cu, let us denote by
$\Lsc_{\mathrm{nc}}(X,S)$ the set of non-compact elements in
$\Lsc(X,S)$.

\begin{remark}
\label{rem:compact} {\rm Observe that, if $X$ is a connected,
compact Hausdorff space, $A$ is a (stable) C$^*$-algebra and $f\sim
p$, for $f\in\CC(X,A)_+$ and $p$ a projection in $\CC(X,A)$, then
$f$ is pointwise equivalent to a projection $q\in A$. This is easy
to verify by a direct argument, but can also be obtained as a
consequence of the fact that, for a semigroup $S$ in Cu, the compact
elements in $\Lsc(X,S)$ are precisely the constant, compact-valued
functions (see, e.g. the arguments in \cite[Corollary 3.8]{aps}). In
particular, such an $f$ is either identically zero or always
non-zero.}
\end{remark}

\begin{lemma}
\label{lem:nc} Let $X$ be compact Hausdorff and connected, and let
$S$ be a semigroup in Cu with cancellation of compact elements and
such that the set of non-compact elements is closed under addition.
Then $\Lsc_{\mathrm{nc}}(X,S)$ is a subsemigroup of $\Lsc(X,S)$.
\end{lemma}
\begin{proof}
Let $f,g\in\Lsc_{\mathrm{nc}}(X,S)$ and assume that $f+g$ is
compact. The arguments in \cite[Corollary 3.8]{aps} show that there
is a compact element $c\in S$ such that $(f+g)(t)=c$ for every $t\in
X$. By our assumptions on $S$, it follows that $f(t)$ and $g(t)$ are
compact for every $t\in X$.

Using that $f(t)\ll f(t)$ and $g(t)\ll g(t)$, and that $f$ and $g$
are lower semicontinuous, find a neighbourhood $U_t$ of $t$ such
that $f(t)\ll f(s)$ and $g(t)\ll g(s)$ for every $s\in U_t$ (see,
e.g. \cite[Lemma 5.1]{aps}). It then follows that
\[
f(t)+g(s)\ll f(s)+g(s)=c=f(t)+g(t)\,.
\]
By cancellation of compact elements, $g(s)\leq g(t)\leq g(s)$ in
$U_t$, so that $g$ is constant in a neighbourhood of $t$. Since $X$
is connected, it follows that $g$ is constant. Likewise, $f$ is
constant.
\end{proof}
When $S$ as above comes as a Cuntz semigroup of a C$^*$-algebra,
then it satisfies the additional axiom of having an ``almost
algebraic order'' (see \cite[Lemma 7.1 (i)]{Rordam-Winter}, and also
\cite{robertfunct}): if $x\leq y$ and $x'\ll x$, then there is $z\in
S$ such that $x'+z\leq y\leq x+z$. One can then prove that, if such
an $S$ has moreover cancellation of compact elements, then the set
$S_{\mathrm{nc}}$ of non-compact elements is a subsemigroup of $S$.
Indeed, if $x+y$ is compact, choose $x'\ll x''\ll x$ such that
$x'+y=x''+y=x+y$. By the almost algebraic order axiom, there is
$z\in S$ with $x'+z\leq x\leq x''+z$. Adding $y$ to this inequality
yields $(x+y)+z\leq x+y$, and since $x+y$ is compact, it follows
that $z=0$, and this implies that $x\leq x''\ll x$.

For a simple, separable C$^*$-algebra with stable rank one, consider
the semigroup
\[
\V(\CC(\T,A))\sqcup\Lsc_{\mathrm{nc}}(\T,\Cu(A))\,,
\]
equipped with addition that extends both the natural operations in
both components, and with
\[
x+f=\hat{x}+f\,,\text{ whenever }x\in \V(\CC(\T,A))\text{ and
}f\in\Lsc_{\mathrm{nc}}(\T,\Cu(A))\,.
\]
We can order this semigroup by taking the algebraic ordering in
$\V(\CC(\T,A))$, the pointwise ordering on
$\Lsc_{\mathrm{nc}}(\T,\Cu(A))$, and we order mixed terms as
follows:
\begin{enumerate}[{\rm (i)}]
\item $f\leq x$ if $f(t)\leq \hat{x}(t)$ for every $t\in \T$.
\item $x\leq f$ if there is $g\in\Lsc_{\mathrm{nc}}(\T,\Cu(A))$ such that
$\hat{x}+g=f$.
\end{enumerate}
That this ordering is transitive is not entirely trivial, but it
follows from the arguments in Theorem \ref{thm:circle} below.

We may now define an ordered map in the category of semigroups:
\[
\begin{array}{ccc} \alpha\colon \Cu_{\T}(A) & \longrightarrow &
\V(C(\T,A))\sqcup \Lsc_{\mathrm{nc}}(\T,\Cu(A)) \\
\mbox{} x & \longmapsto & \left\{ \begin{array}{lr}  x & \text{ if }
x\in
\V(\CC(\T,A))   \\
\hat{x} & \text{ otherwise }
\end{array}\right.\end{array}
\]


\begin{theorem}
\label{thm:circle} If $A$ is a simple C$^*$-algebra with stable rank
one, then there is an order-isomorphism
\[
\Cu_{\T}(A)\cong \V(\CC(\T,A))\sqcup\Lsc_{\mathrm{nc}}(\T,\Cu(A))\,.
\]
\end{theorem}
\begin{proof}
We will show that the map $\alpha$ just defined is a surjective
order-embedding.

First note that $\CC(\T,A)$ is the following pullback
\[
\xymatrix
{\CC(\T,A)\ar[d]\ar[r]^{\mathrm{ev}_1} & A\ar[d] \\
\CC([0,1],A)\ar[r]^{\mathrm{ev}_{0,1}} & A\oplus A}
\]
Since, by Theorem \ref{thm:interval}, the natural map
$\Cu(\CC([0,1],A)\to\Lsc([0,1],\Cu(A))$ is an order-embedding, we
may use \cite[Theorem 3.3]{aps} to conclude that the pullback map
\[
\Cu_{\T}(A)\to \Cu(\CC([0,1],A))\oplus_{\Cu(A\oplus A)}\Cu(A)
\]
is a surjective map in the category Cu. Upon identifying
$\Cu(\CC([0,1],A))\oplus_{\Cu(A\oplus A)}\Cu(A)$ with
$\Lsc(\T,\Cu(A))$, we obtain that the map
\[
\Cu_{\T}(A)\to\Lsc(\T,\Cu(A))\text{, given by } x\mapsto\hat{x}\,,
\]
is also surjective. This implies in particular that the map $\alpha$
is surjective.

To prove that $\alpha$ is an order-embedding, let
$x,y\in\Cu_{\T}(A)$ and assume that $\alpha (x)\leq\alpha(y)$. There
is nothing to prove if $x,y\in \V(\CC(\T,A))$.

If $x\notin \V(\CC(\T,A))$, then write $x=[f]$, $y=[g]$, and our
assumption just means that $f(t)\precsim g(t)$ for every $t\in \T$.
We may then apply Proposition \ref{prop:generalcase} to conclude
that $f\precsim g$.

Finally, assume that $x\in \V(\CC(\T,A))$ and $y\notin
\V(\CC(\T,A))$. Then $\alpha(x)\leq \alpha(y)$ means, by definition,
that there is $g\in\Lsc_{\mathrm{nc}}(\T,\Cu(A))$ with
$\hat{x}+g=\hat{y}$. Let $z\in\Cu_{\T}(A)$ be such that $\hat{z}=g$.
Then
\[
(x+z)^{\hat{}}=\hat{x}+g=\hat{y}\,.
\]
Note that $x+z\notin \V(\CC(\T,A))$, as otherwise $(x+z)^{\hat{}}$
would be a compact element in $\Lsc (\T,\Cu(A))$. By Lemma
\ref{lem:nc} (or rather, its proof -- see also Remark
\ref{rem:compact}), $g=\hat{z}$ would be constant (and compact), a
contradiction.

The argument in the previous paragraph then shows that $x+z=y$, as
wanted.
\end{proof}

\begin{theorem}
\label{thm:zstable} Let $A$ be a separable, finite
$\mathcal{Z}$-stable C$^*$-algebra. Then, there is an
order-isomorphism
\[
\Cu_{\T}(A)\cong (\{0\}\sqcup (\V(A)^*\times
\mathrm{K}_1(A)))\sqcup\Lsc_{\mathrm{nc}}(\T,\Cu(A))\,,
\]
where $\V(A)^*=\V(A)\setminus\{0\}$.
\end{theorem}
\begin{proof}
 By Theorem \ref{thm:circle}, we only need to show that $\V(\CC(\T,A))\cong
\{0\}\sqcup (\V(A)^*\times \mathrm{K}_1(A))$. This follows once we
notice that $\CC(\T,A))$ has cancellation of projections (see, e.g.
\cite{Tikuisis}). Since $A$ is $\mathcal{Z}$-stable, then
$\CC(\T,A)$ is also $\mathcal Z$-stable, whence $\CC(\T,A)$ has
cancellation of full projections by \cite[Theorem 1]{jiang}. We have
already observed (see Remark \ref{rem:compact}) that every
projection in (matrices over) $\CC(\T,A)$ is either identically zero
or always non-zero, and in that case it is a full projection as $A$ is simple, by an application of \cite[Lemma 10.4.2]{dix}.
\end{proof}

\begin{remark}
{\rm In light of these results, one might expect that the same
description of the Cuntz semigroup will hold for more general spaces
(of dimension at most $1$). However, the following example provided by N. C. Phillips shows that this is not the case.

Let $A$ be a simple C$^*$-algebra with stable rank one,
$\mathrm{K}_1(A)\neq 0$ and such that $\V(\CC(\T,A))\cong
\{0\}\sqcup \V(A)^*\times \mathrm{K}_1(A)$ (for example, $A$ could
be $\mathcal{Z}$-stable as above). Let $X=\T\cup[1,2]$, and take
$f'$, $g'\in\CC(\T,A)$ be elements such that $f'(t)\sim g'(t)$ for
all $t\in \T$, yet $f'$ and $g'$ are not comparable. For example, we
could take a non-zero element $[p]\in \V(A)^*$, a non-trivial class
$[u]\in \mathrm{K}_1(A)$, and $f'$ corresponding to $([p],[1])$ and
$g'$ corresponding to $([p],[u])$. Define $f,g\in \CC(X,A)$ as $f',
g'$ over $\T$, and $f(t)=(2-t)f(1)$, $g(t)=g(1)$ for $t\in [1,2]$.
Then clearly $f(t)\precsim g(t)$ for all $t\in X$, but
$f\not\precsim g$.

}
\end{remark}


\section{A categorical approach}

As already shown in \cite{Tikuisis}, the Elliott invariant and
the invariant defined by $\Cu_{\T}(-)$ are equivalent in a
functorial way, for simple, unital non-type I ASH algebras with slow
dimension growth. Because of Theorem \ref{thm:zstable}, this is
actually true in the more general setting of separable $\mathcal
Z$-stable, simple C$^*$-algebras with stable rank one. Our aim in
this section is to develop a (somewhat) abstract approach that makes
the functorial equivalence explicit, thus also proving the Theorem
announced in the Introduction.

Let $S$ be a semigroup in Cu. Assume that the subset
$S_{\mathrm{nc}}$ of non-compact elements is an absorbing
subsemigroup, in the sense that $S_{nc}+S\subseteq S_{\mathrm{nc}}$.
Denote by $S_{\mathrm{c}}$ the subsemigroup of compact elements and
$S_\mathrm{c}^*=S_\mathrm{c}\setminus\{0\}$. Let $G$ be an abelian
group and consider the semigroup
\[
S_G=(\{0\}\sqcup (G\times S_{\mathrm{c}}^*))\sqcup
S_{\mathrm{nc}}\,,
\]
with natural operations in both components, and $(g,x)+y=x+y$
whenever $x\in S_{\mathrm{c}}^*$, $y\in S_{\mathrm{nc}}$, and $g\in
G$. This semigroup can be ordered by
\begin{enumerate}[{\rm (i)}]
\item For $x,y\in S_{\mathrm{c}}^*$, and $g,h\in G$, $(g,x)\leq (h,y)$ if and only if $x=y$ and $g=h$, or else
$x<y$.
\item For $x\in S_{\mathrm{c}}^*$, $y\in S_{\mathrm{nc}}$, $g\in
G$, $(g,x)$ is comparable with $y$ if $x$ is comparable with $y$.
\end{enumerate}
The proof of the following lemma is rather straightforward, hence we
omit the details.
\begin{lemma}
\label{lem:tontada} Let $S$ be an object of Cu such that
$S_{\mathrm{nc}}$ is an absorbing subsemigroup. If $G$ is an abelian
group, then $S_G$ is also an object of Cu.
\end{lemma}

As in \cite{pertoms}, let us write $\mathcal{I}$ to denote the
category whose objects are $4$-tu\-ples
\[
I=((G_0,G_0^+,u),G_1,X,r)\,,
\]
where $(G_0,G_0^+,u)$ is a (countable) simple partially ordered
abelian group with order-unit $u$, $G_1$ is a (countable) abelian
group, $X$ is a (metrizable) Choquet simplex, and $r\colon X\to
\mathrm{S}(G_0,u)$ is an affine map, where $\mathrm{S}(G_0,u)$
denotes the state space of $(G_0,u)$.

Maps between objects $((G_0,G_0^+,u),G_1,X,r)$ and
$((H_0,H_0^+,v),H_1,Y,s)$ of $\mathcal I$ are described as
$3$-tuples $(\theta_0,\theta_1, \gamma)$, where $\theta_0$ is a
morphism of ordered groups with order unit, $\theta_1$ is a morphism
of abelian groups, and $\gamma\colon Y\to X$ is an affine and
continuous map such that $r\circ\gamma=\theta_0^*\circ s$, where
$\theta_0^*\colon \mathrm{S}(H_0,v)\to \mathrm{S}(G_0,u)$ is the
naturally induced map.

Let $\mathcal{C}_s$ denote the class of simple, unital, separable
and nuclear C$^*$-algebras. Then, the Elliott invariant defines a
functor
\[
\mathrm{Ell}\colon\mathcal{C}_s\to \mathcal I
\]
by
\[
\mathrm{Ell}(A)=((\K_0(A), \K_0(A)^+, [1_A]), \K_1(A),
\mathrm{T}(A),r)\,,
\]
where $\mathrm{T}(A)$ is the trace simplex and $r$ is the pairing
between $\K$-Theory and traces.

Let us define a functor
\[
F\colon \mathcal{I}\to \Cu
\]
as follows. If $I=((G_0,G_0^+,u),G_1,X,r)$ is an object of
$\mathcal{I}$, set
\[
F(I)=(\{0\}\sqcup (G_1\times G_0^{++}))\sqcup\Lsc_{\mathrm{nc}}(\T,
G_0^+\sqcup \mathrm{LAff}(X)^{++})\,,
\]
where $G_0^{++}=G_0^+\setminus\{0\}$.

Since $G_0^+\sqcup \mathrm{LAff}(X)^{++}$ is an object of Cu (see,
e.g. \cite[Lemma 6.3]{abp}), it follows from Lemma \ref{lem:tontada}
above that $F(I)$ is also an object of Cu. (The addition on
$G_0^+\sqcup \mathrm{LAff}(X)^{++}$ is given by
$(g+f)(x)=r(x)(g)+f(x)$, where $g\in G$, $f\in\mathrm{LAff}(X)$ and
$x\in X$.)

That $F$ is a functor follows almost by definition. The only
non-trivial detail that needs to be checked is that, if
\[
(\theta_0,\theta_1,\gamma)\colon((G_0,G_0^+,u),G_1,X,r)\to((H_0,H_0^+,v),H_1,Y,s)
\]
is a morphism in $\mathcal{I}$ and $f\colon\T\to
G_0^+\sqcup\mathrm{LAff} (X)^{++}$ is non-compact, then
$(\theta_0\sqcup\gamma^*)\circ f\colon\T\to
H_0^+\sqcup\mathrm{LAff}(Y)^{++}$ is also non-compact. Here
\[
\theta_0\sqcup\gamma^*\colon G_0^+\sqcup\mathrm{LAff}(X)^{++}\to
H_0^+\sqcup\mathrm{LAff}(Y)^{++}
\]
is defined as $\theta_0$ on $G_0^+$ and $\gamma^*$ on
$\mathrm{LAff}(X)^{++}$.

If $(\theta_0\sqcup\gamma^*)\circ f$ is compact, then there is $h\in
H_0^+$ such that $\theta_0(f(\T))=\{h\}$ and $f(\T)\subseteq G_0^+$.
As $f$ is non-compact and lower semicontinuous, there are $s,t\in
\T$ with $f(t)<f(s)$, whence $f(s)-f(t)\in G_0^{++}$ is an
order-unit. Thus, there exists $n\in\N$ with $f(s)\leq
n(f(s)-f(t))$. After applying $\theta_0$, we obtain that $h\leq 0$,
so that $h=0$. But this is not possible since, as $f$ is not
constant, it takes some non-zero value $a$, which will be an
order-unit with $\theta_0(a)=0$, contradicting that $\theta_0(u)=v$.

Let us show that $F\colon\mathcal{I}\to F(\mathcal I)$ is a full,
faithful and dense functor, so it yields an equivalence of
categories. Therefore, by standard category theory, there exists a
functor $G\colon F(\mathcal I)\to \mathcal I$ such that $F\circ G$
and $G\circ F$ are naturally equivalent to the (respective)
identities.

We only need to prove that $F$ is a faithful functor. If
\[
(\theta_0,\theta_1,\gamma)\colon ((G_0,G_0^+,u),G_1,X,r)\to
((H_0,H_0^+,v),H_1,Y,s)
\]
is a morphism in $\mathcal I$, we shall write
$F((\theta_0,\theta_1,\gamma))=(\theta_1\times\theta_0)\sqcup(\theta_0\sqcup(\gamma^*)_*)$,
where
\[
(\theta_0\sqcup(\gamma^*))_*(f)=(\theta_0\sqcup(\gamma^*))\circ f\,,
\]
for $f\in \Lsc_{\mathrm{nc}}(\T,G_0^+\sqcup\mathrm{LAff}(X)^{++})$.
If now
$F((\theta_0,\theta_1,\gamma))=F((\theta'_0,\theta'_1,\gamma'))$, we
readily see that $\theta_0\times\theta_1=\theta'_0\times\theta'_1$,
whence $\theta_i=\theta'_i$. It also follows that
$\gamma^*(h)=h\circ\gamma=h\circ\gamma'=\gamma'^*(h)$, for every
affine continuous function $h$ on $X$. Since $X$ is homeomorphic to
the state space on $\mathrm{Aff}(X)$ (normalized at the constant
function $1$) via the natural evaluation map $\psi\colon X\to
\mathrm{S}(\mathrm{Aff}(X),1)$ (e.g. \cite[Theorem 7.1]{poag}), the
compositions
\[
\xymatrix{
 {Y}\ar[r]^-{\cong}
                & {\mathrm{S}(\mathrm{Aff}(Y),1)}\ar@<0.5ex>[r]^-{\gamma^*} \ar@<-0.5ex>[r]_-{\gamma'^{*}}&{{\mathrm{S}(\mathrm{Aff}(X),1)}}\ar[r]^-{\cong}&X}
                \]
yield that $\gamma=\gamma'$.

Assembling our observations above (together with Theorem
\ref{thm:zstable} and \cite[Corollary 5.7]{bpt}), we get the
following:

\begin{theorem}
\label{thm:equivalence} {\rm (Cf. \cite{Tikuisis})} Upon restriction
to the class of unital, simple, separable and finite $\mathcal
Z$-stable algebras, there are natural equivalences of functors
\[
F\circ\mathrm{Ell}\simeq\Cu_{\T}\,\text{ and }\,\mathrm{Ell}\simeq
G\circ\Cu_{\T}\,.
\]
Therefore, for these algebras, $\mathrm{Ell}$ is a classifying
functor if, and only if, so is $\Cu_{\T}$.
\end{theorem}

\section*{Acknowledgements}
This work was carried out at the Centre de Recerca Matem\` atica (Bellaterra) during the programme ``The Cuntz Semigroup and the Classification of C$^*$-algebras'' in 2011. We gratefully acknowledge the support and hospitality extended to us. It is also a pleasure to thank N. Brown, I. Hirshberg, N. C. Phillips and H. Thiel for interesting discussions concerning the subject matter of this paper. The first, third and fourth authors were partially supported by a MEC-DGESIC grant (Spain) through Project MTM2008-06201-C02-01/MTM, and by the Comissionat per Universitats i Recerca de la Generalitat de Catalunya.
The second author  was partially supported by
NSF grant \#DMS--1101305.

\end{document}